\documentclass[10pt]{amsart}
\usepackage{amsmath}
\usepackage{amssymb} 
\usepackage{epsbox}
\usepackage{amsthm}
\usepackage{graphicx}
\usepackage[all]{xy}

\begin{document}  
\title[Alexander criterion for bi-orderability]{A remark on Alexander polynomial criterion for bi-orderability of fibered 3-manifold groups}
\author{Tetsuya Ito}
\address{Graduate School of Mathematical Science, University of Tokyo, 3-8-1 Komaba, Meguro-ku, Tokyo, 153-8914, Japan}
\email{tetitoh@ms.u-tokyo.ac.jp}
\urladdr{http://ms.u-tokyo.ac.jp/~tetitoh}

 
\newtheorem{thm}{Theorem}
\newtheorem{ethm}{Expected generalization}
\newtheorem{cor}{Corollary}
\newtheorem{lem}{Lemma}
\newtheorem{prop}{Proposition}

 {\theoremstyle{definition}
  \newtheorem{defn}{Definition}}
 {\theoremstyle{definition}
  \newtheorem{exam}{Example}}
 {\theoremstyle{definition}
  \newtheorem{rem}{Remark}}

\newcommand{\Z}{\mathbb{Z}}
\newcommand{\Q}{\mathbb{Q}}
\newcommand{\GL}{\textrm{GL}}
\newcommand{\tF}{\widetilde{F}}
\newcommand{\tM}{\widetilde{M}}
\newcommand{\tS}{\widetilde{\Sigma}}
\newcommand{\ttheta}{\widetilde{\theta}}


\begin{abstract}
We observe that Clay-Rolfsen's obstruction of bi-orderability, which uses the classical Alexander polynomial, is not strengthened by using the twisted Alexander polynomials for finite representations unlike many known applications of the Alexander polynomial. This is shown by studying the maximal ordered abelian quotient of bi-ordered groups. 
\end{abstract}
\maketitle

\section{Introduction}

A total ordering $<$ of a group $G$ is a {\it bi-ordering} if the ordering $<$ is invariant under the multiplication of $G$ from both sides. That is, $f<g$ implies $hf <hg$ and $fh<gh$ for all $f,g,h \in G$. A group $G$ is called {\it bi-orderable} if $G$ has a bi-ordering.

Recently there are progresses on the problem to determine which low-dimensional manifold has the bi-orderable fundamental group. For $2$-dimensional manifolds and Seifert fibered 3-manifolds, this bi-orderability problem is completely solved in \cite{brw}. In this paper we consider the bi-orderability problem for fibered 3-manifold groups.

In \cite{pr},\cite{pr2}, Perron and Rolfsen showed that for a fibered 3-manifold $M$, if the all roots of the classical Alexander polynomial are real and positive, then the fundamental group of $M$ is bi-orderable. Recently, Clay and Rolfsen showed the partial converse.

\begin{thm}[Clay-Rolfsen's obstruction \cite{cr}]
\label{thm:clay-rolfsen}
Let $M$ be an orientable fibered 3-manifold and $\phi: \pi_{1}(M) \rightarrow \Z$ be a surjective homomorphism induced by a fibration map $M \rightarrow S^{1}$. If $\pi_{1}(M)$ is bi-orderable, then the Alexander polynomial $\Delta_{M}^{\phi}(t)$ has at least one positive real root.   
\end{thm}

There is a generalization of the Alexander polynomial called the {\it twisted Alexander polynomial}, which is defined for each finite dimensional representation of the fundamental group.
Many results on 3-manifolds which uses the Alexander polynomial can be generalized by using twisted Alexander polynomials in a rather simple way (See section 3.3). Moreover, in most cases generalized arguments are stronger than the original argument. For example,  in \cite{fv2} it is shown that the twisted Alexander polynomial gives necessary and sufficient condition for 3-manifolds to be fibered by generalizing the classical Alexander polynomial criterion. 

Therefore it is natural to try to generalize Clay-Rolfsen's result for the twisted Alexander polynomials. The aim of this paper is to show the negative result: The twisted Alexander polynomial cannot be used to strengthen Clay-Rolfsen's obstruction. 

\begin{thm}
\label{thm:main-alex}
Let $M$ be a fibered 3-manifold whose fundamental group is bi-orderable. 
Then the twisted Alexander polynomial for a finite representation contains the same or less information about bi-ordering of $\pi_{1}(M)$ compared with the classical Alexander polynomial. 
\end{thm}

During the proof of Theorem \ref{thm:main-alex} which will be given in section 3, we will clarify which factor of the twisted Alexander polynomial contains the information of bi-orderings. We will show that this essential factor of the twisted Alexander polynomial is determined by the essential part of the classical Alexander polynomials. 

To prove Theorem \ref{thm:main-alex}, in Section 2 we study the {\it maximal ordered abelian quotient} of bi-ordered groups and show the following result.

\begin{thm}
\label{thm:main}
Let $(G,<_{G})$ be a bi-ordered group and $H$ be a finite index subgroup of $G$. Let $<_{H}$ be a bi-ordering of $H$ defined by the restriction of the ordering $<_{G}$ to $H$. Then the natural inclusion map of bi-ordered groups $i: (H,<_{H}) \hookrightarrow (G,<_{G})$ induces an isomorphism of $\Q$-vector space 
\[ A(H,<_{H}) \otimes \Q \rightarrow A(G,<_{G}) \otimes \Q. \]
\end{thm}

This result is interesting in its own right. Theorem \ref{thm:main} represents a rigidity of bi-ordered groups, and shows that the behavior of the rank of the maximal ordered abelian quotient is quite different from the rank of the maximal abelian quotient.  

Though our main result is a negative result, but is still interesting because it provides an example of the non-effectiveness of twisted Alexander polynomials. 
Moreover, Theorem \ref{thm:main} shows that taking a finite cover, one of the most useful techniques to study 3-manifolds, is not useful to study bi-orderability. 
This clarifies the difficulty of the bi-orderability problem in low-dimensional manifold groups. \\

\textbf{Acknowledgments.}
The author would like to thank Takahiro Kitayama and Hiroshi Goda for many useful conversation and advices for twisted Alexander polynomials. He also wishes to express his thanks to Dale Rolfsen and Adam Clay for introducing their results to the author.
This research was supported by JSPS Research Fellowships for Young Scientists.

\section{The maximal ordered abelian quotient}

In this section we study the maximal ordered abelian quotient of bi-ordered group $(G,<_{G})$. 

A subgroup $H$ of $G$ is {\it convex} with respect to the ordering $<_{G}$ if for $h,h' \in H$ and $g \in G$, the inequality $h <_{G} g <_{G} h'$ implies $g \in H$. 
Assume that $H$ is a convex, normal subgroup.
Then the quotient group $G \slash H$ has a bi-ordering $<_{G \slash H}$ induced by the ordering $<_{G}$, defined by
$xH <_{G \slash H} yH$ if $x<_{G}y$ where $x,y \in G$ are arbitrary chosen coset representatives. 
We say the bi-ordered group $(G\slash H, <_{G\slash H})$ {\it the ordered quotient group}.

Now we introduce the {\it convex commutator subgroup}, which is a refinement of the commutator subgroup in the category of bi-ordered groups.
\begin{defn}
 The {\it convex commutator subgroup} of a bi-ordered group $(G,<)$ is a convex subgroup $C_{G}$ defined as the intersection of all convex subgroups of $(G,<_{G})$ which contains $[G,G]$.
\end{defn}

Since intersection of convex subgroups is a convex subgroup, $C_{G}$ is the minimal convex subgroup of $(G,<_{G})$ which contains the commutator subgroup $[G,G]$. First we show basic properties of $C_{G}$.

\begin{lem}
\label{lem:conv}
Let $(G,<_{G})$, $(H,<_{H})$ be bi-ordered groups and $\theta: (G,<_{G}) \rightarrow (H,<_{H})$ be an order-preserving homomorphism. 
\begin{enumerate}
\item $\theta(C_{G}) \subset C_{H}$.
\item $C_{G}$ is a normal subgroup of $G$. 
\end{enumerate}
\end{lem}
\begin{proof}
Let $N$ be the intersection of all convex subgroups of $(H,<_{H})$ which contains $\theta([G,G])$. Since $\theta([G,G]) \subset [H,H]$, $N \subset C_{H}$.
$\theta$ is order-preserving, so $\theta(C_{G})$ is a convex subgroup of $(H,<_{H})$ which contains $\theta([G,G])$. Hence $\theta(C_{G}) \supset N$. If $N$ is a proper subgroup of $\theta(C_{G})$, $\theta^{-1}(N)$ is a proper convex subgroup of $C_{G}$ which contains $[G,G]$, so it is a contradiction. Thus, we conclude $\theta(C_{G}) = N \subset C_{H}$.

To show (2), observe that every inner automorphism of $G$ preserves a bi-ordering $<_{G}$. Hence by (1), $C_{G}$ is preserved by all inner automorphisms, so $C_{G}$ is normal.
\end{proof}
  
Now we are ready to define the maximal ordered abelian quotient.
\begin{defn}
For a bi-ordered group $(G,<_{G})$, {\it the maximal ordered abelian quotient group} of $(G,<)$ is an ordered quotient group
\[ A(G,<_{G}) = (G\slash C_{G}, <_{G \slash C_{G}}). \] 
\end{defn}

$A(G,<_{G})$ plays a similar role of the maximal abelian quotient in the category of bi-ordered groups.
 
\begin{lem}
\label{lem:maxoabel}
Let $(G,<_{G})$, $(H,<_{H})$ be bi-ordered groups and $\theta: (G,<_{G}) \rightarrow (H,<_{H})$ be an order-preserving homomorphism.
\begin{enumerate}
\item $\theta$ induces a well-defined order-preserving homomorphism $\theta_{*}: A(G,<_{G}) \rightarrow A(H,<_{H})$. 
\item There exists a direct sum decomposition of the $\Q$-vector space $H_{1}(G;\Q) = (A(G,<_{G}) \otimes \Q) \oplus V_{G}$ which is preserved by any order-preserving automorphisms of $G$. 
\item $A(G,<_{G})$ is non-trivial.
\end{enumerate}
\end{lem}
\begin{proof}
Assertion (1) is a direct consequence of Lemma \ref{lem:conv}.
To show assertion (2), let $V_{G}$ be the kernel of the map $p \otimes id_{\Q} : H_{1}(G;\Z) \otimes \Q \rightarrow A(G,<_{G}) \otimes \Q $ where $p: H_{1}(G;\Z)= G\slash [G,G] \rightarrow A = G \slash C_{G}$ is the natural projection. Since $C_{G}$ is preserved by any order-preserving automorphisms, so is $A(G,<_{G}) \otimes \Q$. Hence the direct sum decomposition $H_{1}(G;\Q) = (A(G,<_{G}) \otimes \Q) \oplus V_{G}$ is invariant under order-preserving automorphisms.

Finally, assertion (3) follows from \cite[Lemma 2.2]{cr}, which asserts that there exists the maximal proper convex subgroup $C$ of $(G,<_{G})$, and that $G\slash C$ is abelian. Since $C \supset C_{G}$, this implies $A(G,<_{G})$ is non-trivial.
\end{proof}

We proceed to study more properties of the maximal ordered abelian quotient.
Recall that the commutator $[a,b] = a^{-1}b^{-1}ab$ satisfies commutator identities
\[ [a,bc] = [a,c][a,b][[a,b],c], \; [ab,c] = [a,c][[a,c],b][b,c] .\]

These identities give the following inequalities. 

\begin{lem}
\label{lem:comm}
Let $(G,<_{G})$ be a bi-ordered group and $a, b \in G$.
\begin{enumerate}
\item $[a,b] <_{G} b$ if $b >_{G} 1 $ and $[a,b] >_{G} b$ if $b<_{G}1$.
\item $[a,b] >_{G} a^{-1}$ if $a>_{G}1$ and $[a,b] <_{G} a^{-1}$ if $a<_{G}1$.
\item If $[a,b]>_{G} 1$, then $[a^{n},b^{m}] >_{G} [a,b]$ holds for $m,n>1$.
\end{enumerate}
\end{lem}
\begin{proof}
A proof of (1), (2) is routine. We show (3).
First we show $[a,b^{m}] >_{G} [a,b]$ by induction on $m$. 
By (2), we have $[[a,b],b^{m-1}] >_{G} [a,b]^{-1}$ and
from inductive hypothesis we have $[a, b^{m-1}] >_{G} [a,b]$. Thus, by commutator identity we obtain an inequality
\[ [a,b^{m}] = [a, b^{m-1}][a,b][[a,b],b^{m-1}] >_{G} [a,b][a,b][a,b]^{-1} = [a,b].\]
Similarly, by induction on $n$ we get an inequality $ [a^{n},b] >_{G} [a,b]$.
These two inequalities give the desired inequality.
\end{proof}

Now we are ready to prove Theorem \ref{thm:main}.

\begin{proof}[Proof of Theorem \ref{thm:main}]
Since $[G:H]$ is finite, there is a finite integer $N>0$ such that $g^{N} \in H$ holds for all $g \in G$.
Thus by Lemma \ref{lem:comm} (3), for any commutator $[g,g']$ of elements in $G$, we can find a commutator $[h,h']$ of elements in $H$ which satisfies the inequality
\[ [h,h']^{-1} <_{G} [g,g'] <_{G} [h,h'].\]
This implies that $C_{H} = C_{G} \cap H$. Therefore we get a sequence of isomorphisms of $\Q$-vector spaces
\[ (H\slash C_{H}) \otimes \Q \cong (H \slash C_{G} \cap H) \otimes \Q \cong (HC_{G} \slash C_{G}) \otimes \Q \cong (G\slash C_{G}) \otimes \Q .\]
\end{proof}

\section{Why the twisted Alexander polynomials are useless ?}

In this section we explain why the twisted Alexander polynomials cannot make Clay-Rolfsen's obstruction powerful.

\subsection{Twisted Alexander polynomial}

We review the definition and basic properties of the twisted Alexander polynomials. For details, see \cite{fv3}.

Let $\phi: \pi_{1}(M) \rightarrow \Z  = \langle t \rangle $ be a non-trivial homomorphism and $V_{\alpha}$ be a finite dimensional left $\Q \pi_{1}(M)$-module defined by the presentation $\alpha: \pi_{1}(M) \rightarrow \GL(V)$.
We say $\alpha$ is a {\it finite representation} if the image of $\alpha$ is finite. The classical Alexander polynomial is obtained as the twisted Alexander polynomial corresponding to the trivial representation $\varepsilon: \pi_{1}(M) \rightarrow \GL(\Q)=\Q$.

Let $\alpha \otimes \phi: \pi_{1}(M) \rightarrow \GL(V\otimes_{\Q} \Q[t,t^{-1}])$ be the product representation given by 
$[ \alpha \otimes \phi ](g): v \otimes t^{i} \mapsto [\alpha(g)](v)\otimes t^{i+\phi(g)}$.
Then $V_{\alpha \otimes \phi}$ is a left $\Q \pi_{1}(M)$-module and the action of $\Q\pi_{1}(M)$ commutes with the right action of $\Q[t,t^{-1}]$.

The {\it $i$-th twisted Alexander module} is the $\Q[t,t^{-1}]$-module defined by the $i$-th twisted coefficient homology group
\[ H_{i}(M;V_{\alpha \otimes \phi}) = H_{i} (C_{*}(\widetilde{M}) \otimes_{\Q \pi_{1}(X)} V_{\alpha \otimes \phi}) \]
where $C_{*}(\widetilde{M})$ is the singular chain complex of the universal cover $\widetilde{M}$ of $M$, viewed as a right $\Q \pi_{1}(M)$-module.

Since $H_{i}(M;V_{\alpha \otimes \phi})$ is a finitely generated $\Q[t,t^{-1}]$-module, there exist elements $p_{i}(t) \in \Q[t,t^{-1}]$ and an isomorphism as $\Q[t,t^{-1}]$-module
\[ H_{i}(M;V_{\alpha \otimes \phi}) \cong \Q[t,t^{-1}]^{k} \oplus \bigoplus_{i=1}^{m} \Q[t,t^{-1}]\slash ( p_{i}(t) ). \]

The elements $p_{i}(t)$ are well-defined up to multiplication by a unit of $\Q[t,t^{-1}]$ if we add the condition that $p_{i}(t)$ divides $p_{i+1}(t)$ for each $i$.
The {\it $i$-th twisted Alexander polynomial} $\Delta_{M,i}^{\alpha \otimes \phi}$ is a Laurant polynomial defined by
\[ \Delta_{M,i}^{\alpha \otimes \phi}(t) = \left\{ 
\begin{array}{ll} 
\prod_{i=1}^{m} p_{i}(t) & k=0 \\
0 & k\neq 0.
\end{array}
\right.
\]

The twisted Alexander polynomials are well-defined up to multiplication by a unit of $\Q[t,t^{-1}]$. In particular, the (non-zero) roots of the twisted Alexander polynomial are well-defined. 
For 3-manifolds, we only consider the $1$st twisted Alexander polynomial and simply denote the $1$st twisted Alexander polynomial by $\Delta_{M}^{\alpha \otimes \phi}$.

The twisted Alexander polynomials has the following properties. 
\begin{lem}
\label{lem:changeabel}
Let $d \cdot \phi: \pi_{1}(M) \rightarrow \Z = \langle t \rangle$ be the homomorphism defined by $d\cdot\phi(g) = t^{\phi(g) \cdot d}$. Then
\[ \Delta_{M}^{\alpha \otimes d\cdot \phi}(t) = \Delta_{M}^{\alpha \otimes \phi}(t^{d})\]
holds.
\end{lem}
\begin{lem}
\label{lem:sumrep}
Let $\alpha: \pi_{1}(M) \rightarrow \GL(V)$, $\beta:  \pi_{1}(M) \rightarrow \GL(W)$ be finite dimensional representations, and $\alpha \oplus \beta:  \pi_{1}(M) \rightarrow \GL(V\oplus W)$ be its direct sum. Then 
\[  \Delta_{M}^{(\alpha \oplus \beta) \otimes \phi}(t) = \Delta_{M}^{\alpha \otimes \phi}(t) \cdot \Delta_{M}^{\beta \otimes \phi}(t)\]
holds.
\end{lem}

Now we prove (a special case of) Shapiro's lemma for twisted Alexander polynomials.
Let $G$ be a finite group and $p:\tM_{G} \rightarrow M$ be a $G$-covering of $M$ which corresponds to a certain surjective homomorphism $f: \pi_{1}(M) \rightarrow G$.
Let $\phi: \pi_{1}(M) \rightarrow \Z$ be a non-trivial homomorphism and $p^{*}\phi: \pi_{1}(\tM_{G}) \rightarrow \Z$ be its pull-back.
Let $\alpha = \iota \circ f: \pi_{1}(M) \rightarrow \GL( \Q G)$ be a finite representation where $\iota : G \rightarrow \GL(\Q G)$ be the regular representation of $G$. We call such a representation a {\it regular finite representation}.

\begin{lem}[Shapiro's Lemma \cite{fv3}]
\label{lem:shapiro}
Let $\alpha: \pi_{1}(M) \rightarrow \GL( \Q G)$ be a regular finite representation. Then there is an equality
\[ \Delta_{M}^{\alpha \otimes \phi}(t) = \Delta_{\tM_{G}}^{p^{*}\phi}(t) \]
where $\Delta^{p^{*}\phi}_{\tM_{G}}$ is the classical Alexander polynomial of $\tM_{G}$ with respect to $p^{*}\phi$.
\end{lem}
\begin{proof}
Let $\widetilde{M}$ be the common universal covering of $M$ and $\tM_{G}$.
We have the isomorphisms of the chain complexes
\begin{eqnarray*}
 C_{*}(\widetilde{M}) \otimes_{ \Q \pi_{1}(\tM_{G})} V_{\alpha \otimes p^{*}\phi} & = &  C_{*}(\widetilde{M}) \otimes_{ \Q \pi_{1}(M)} (\Q \pi_{1}(M) \otimes_{ \Q \pi_{1}(\tM_{G})} V _{p_{*}\alpha \otimes \phi}) \\
 & = & C_{*}(\widetilde{M}) \otimes_{ \Q \pi_{1}(M)} V_{ p_{*}\alpha \otimes \phi}. 
 \end{eqnarray*}
 Hence we obtain isomorphisms of the twisted Alexander modules and the desired equality of the twisted Alexander polynomials.
 \end{proof}
 
\subsection{Strategy to get stronger condition from classical Alexander polynomials}

Assume that we have already obtained some arguments which uses the classical Alexander polynomial. That is, we have a statement of the form {\it ``If a 3-manifold $M$ has a property $X$, then its Alexander polynomial has a property $Y$."}
Then we extend the argument for the twisted Alexander polynomials for finite representation according to the following strategy, as in \cite{fv1}. 

\begin{enumerate}
\item First consider the twisted Alexander polynomials for regular finite representations. By Shapiro's lemma, such twisted Alexander polynomials are the classical Alexander polynomials of finite covering.

\item Verify that the property $X$ of 3-manifolds we are considering is preserved by taking finite coverings. 
(For example, the property that the fundamental group is bi-orderable is preserved by taking finite coverings.)

\item By the classical Alexander polynomial argument, a twisted Alexander polynomial for a regular finite representation has the property $Y$ .

\item Consider the irreducible decomposition of the regular representation and get the corresponding factorization of the twisted Alexander polynomials (See Remark \ref{rem:rep} below).

\item Study how the property $Y$ behaves under the factorization and obtain the property of twisted Alexander polynomials. For example, if each factor of the polynomial also has the property $Y$, then we conclude that every twisted Alexander polynomial for finite representation has the property $Y$.
\end{enumerate}

\begin{rem}
\label{rem:rep}
Since we are studying the representation over $\Q$, by Maschke's theorem representations of a finite group $G$ over $\Q$ are completely reducible, and each irreducible representation appears as an irreducible summands of the regular representation. Thus by Lemma \ref{lem:sumrep}, we can obtain all twisted Alexander polynomials for finite representations from the twisted Alexander polynomials for regular finite representations. 
\end{rem}

\subsection{Review of Clay-Rolfsen's argument}

Now we briefly review Clay-Rolfsen's argument. Throughout the rest of this paper, we always assume that 3-manifold $M$ is fibered, and a homomorphism $\phi: \pi_{1}(M) \rightarrow \Z$ is derived from a fibration map $M \rightarrow S^{1}$. We also put $\theta: \Sigma \rightarrow \Sigma$ be the monodromy map and  $F=\pi_{1}(\Sigma)$.

The Clay-Rolfsen's argument is based on the following well-known fact of an HNN-extension of bi-orderable groups.

\begin{lem}
\label{lem:HNN}
Let $H$ be a bi-orderable group and $G$ be an HNN-extension of $H$ by the automorphism $\phi$. Then $G$ is bi-orderable if and only if there exists a bi-ordering of $H$ which is preserved by $\phi$. 
\end{lem}

Assume that $\pi_{1}(M)$ is bi-orderable.
Then by Lemma \ref{lem:HNN}, there exists a bi-ordering $<_{F}$ of $F$ which is invariant under the monodromy $\theta_{*}: F \rightarrow F$.
So $\theta$ induces an order-preserving map $\theta_{*}^{A} : A(F,<_{F}) \rightarrow A(F,<_{F})$ by Lemma \ref{lem:conv}.
Let $\chi_{A}(t)$ be the characteristic polynomial of $\theta_{*}^{A}\otimes id_{\Q}$. 
The key lemma shown in \cite{cr} is the following.

\begin{lem}
\label{lem:key}
Let $(A,<_{A})$ be a bi-ordered abelian group of finite rank and $\theta : A \rightarrow A$ be an order-preserving automorphism. 
Then the $\Q$-linear map $\theta \otimes id_{\Q} : A \otimes \Q \rightarrow A \otimes \Q$ has at least one positive real eigenvalue.
\end{lem}

Thus, $\chi_{A}(t)$ has at least one positive real root. 
The classical Alexander polynomial $\Delta_{M}^{\phi}(t)$ is a characteristic polynomial of $\theta_{*} : H_{1}(F;\Q) \rightarrow H_{1}(F;\Q)$. By Lemma \ref{lem:maxoabel}, we have a $\theta_{*}$-invariant direct sum decomposition $H_{1}(F;\Q) = V_{G} \oplus (A \otimes \Q)$, hence $\Delta_{M}^{\phi}(t)$ divides $\chi_{A}(t)$. Therefore we conclude that $\Delta_{M}^{\phi}(t)$ has at least one positive real root.

The above argument shows that the important part of the Alexander polynomial which contains information about bi-ordering is not the Alexander polynomial itself, but its factor $\chi_{A}(t)$, the characteristic polynomial of the monodromy map $\theta_{*}^{A}$ induced on the maximal ordered abelian quotient.

\subsection{The failure of twisted Alexander polynomial argument}

Now we are ready to explain why we cannot get better obstruction by using the twisted Alexander polynomials. According to the strategy described in section 3.2, we expect the following result.

\begin{ethm}
Let $\rho: \pi_{1}(M) \rightarrow \GL (V)$ be a finite representation. If $\pi_{1}(M)$ is bi-orderable, then the twisted Alexander polynomial $\Delta_{M}^{\rho\otimes \phi}(t)$ has at least one positive real root.   
\end{ethm}

However, it is easy to see such a generalization is impossible because the property that a polynomial has at least one positive real root is not preserved under the factorization of the polynomial. so the step 5 in our strategy does not work. 
Now one might expect another generalization, which is more likely to hold.

\begin{ethm}
Let $\rho: \pi_{1}(M) \rightarrow \GL (\Q G)$ be a regular finite presentation. If $\pi_{1}(M)$ is bi-orderable, then the twisted Alexander polynomial $\Delta_{M}^{\rho\otimes \phi}(t)$ has more positive real roots than the classical Alexander polynomial. 
\end{ethm}

However as the following example shows, this is also not true.

\begin{exam}[The simplest counter example]
Let $K$ be the figure-eight knot. $K$ is fibered, and its Alexander polynomial $\Delta_{K}(t)=t^{2}-3t+1$ has two positive real roots. Thus, $\pi_{1}(S^{3}-K)$ is bi-orderable by Perron-Rolfsen's criterion \cite{pr}.
Now let us consider the regular representation $\alpha = \pi_{1}(S^{3}-K) \rightarrow \GL(\Q  \Z_{2})$. Then $\Delta_{K}^{\alpha} = (t^{2}+3t+1)(t^{2}-3t+1)$, which has exactly the same positive real roots as the classical Alexander polynomial.
\end{exam}

Now it is easy to see why these expected generalization is impossible.

\begin{proof}[Proof of Theorem \ref{thm:main-alex}]

Let $\beta: \pi_{1}(M) \rightarrow G \rightarrow \GL (V)$ be a finite representation and $\alpha: \pi_{1}(M) \rightarrow G \rightarrow \GL (\Q G)$ be a regular finite representation.
By Lemma \ref{lem:sumrep}, the twisted Alexander polynomial $\Delta^{\beta\otimes \phi}_{M}(t)$ is a factor of $\Delta^{\alpha \otimes \phi}(t) ^{k}$ for sufficiently large $k>0$. See Remark \ref{rem:rep} again. Thus, it is sufficient to consider the regular finite representation $\alpha$.

Assume that $\pi_{1}(M)$ has a bi-ordering $<_{M}$.  
Let $\tM$ be the corresponding $G$-cover and $\widetilde{\phi} : \pi_{1}(\tM) \rightarrow \Z$ be the surjective homomorphism induced by the induced fibration map $\tM \rightarrow S^{1}$. This map is not the same as the pull-back homomorphism $p^{*}\phi$, but there exists an integer $d$ such that $p^{*}\phi = d\cdot\widetilde{\phi}$.

Let $\Sigma$, $\tS$ be the fiber of $M$, $\tM$ respectively and put $F = \pi_{1}(\Sigma)$, $\tF= \pi_{1}(\tS)$. 
$\tS$ is also a regular finite covering of $\Sigma$, hence $\tF$ is normal subgroup of $F$ having finite index in $F$. We regard all of $F$, $\tF$ and $\pi_{1}(\tM)$ as subgroups of $\pi_{1}(M)$ and let $<_{F},<_{\tF}$ be the restrictions of the bi-ordering $<_{M}$ to $F,\tF$ respectively. Then by Lemma \ref{lem:key}, the orderings $<_{F}$, $<_{\tF}$ are preserved by the monodromy map $\theta$, $\ttheta$ respectively.

By Shapiro's lemma, the twisted Alexander polynomial associated to $\alpha$ is given as the classical Alexander polynomial 
\[ \Delta^{\alpha \otimes \phi}_{M}(t)= \Delta^{\varepsilon \otimes p^{*}\phi}_{\tM}(t). \]
On the other hand, since $p^{*}\phi = d\cdot\widetilde{\phi}$, by Lemma \ref{lem:changeabel} we get an equality
\[ \Delta^{\alpha \otimes \phi}_{M}(t)= \Delta^{\varepsilon \otimes p^{*}\phi}_{\tM}(t) = \Delta^{\widetilde{\phi}}_{\tM}(t^{d}). \]
For fibered 3-manifolds, the classical Alexander polynomial is given as the characteristic polynomial of the monodromy. So finally we get a description of twisted Alexander polynomial as the characteristic polynomial of the monodromy $\ttheta$,
\[ \Delta^{\alpha \otimes \phi}_{M}(t) =\Delta^{\widetilde{\phi}}_{\tM}(t^{d}) = det(t^{d} I -  \ttheta_{*} ). \]

Now from Clay-Rolfsen's argument, the factor of $\Delta^{\alpha \otimes \phi}_{M}(t)$ which contains information about bi-ordering is the characteristic polynomial of the induced map 
\[ \ttheta_{*}^{A}: A(\tF;<_{\tF}) \rightarrow A(\tF; <_{\tF}). \]
Since $\tF$ is a finite index subgroup of $F$, by Theorem \ref{thm:main} we get a commutative diagram
\[ \xymatrix{
 A(\tF;<_{\tF})\otimes \Q \ar[r]^{\ttheta_{*}^{A}} \ar[d]^{\cong}& A(\tF; <_{\tF})\otimes \Q \ar[d]_{\cong} \\
 A(F;<_{F}) \otimes \Q\ar[r]^{\theta_{*}^{A}} & A(F; <_{F})\otimes \Q
}
\] 
where the vertical arrows are isomorphisms.
Thus, the essential factor of the twisted Alexander polynomial $\Delta^{\alpha \otimes \phi}_{M}(t)$ which has information of the bi-ordering $<_{M}$, equals to the essential part of the classical Alexander polynomial, which gives rise to Clay-Rolfsen's obstruction. 

\end{proof}

\end{document}